\documentclass[12pt]{article}

\usepackage{amsmath}
\usepackage{amssymb}
\usepackage{amsthm}
\usepackage{graphicx}
\usepackage{amscd}
\usepackage{epic, eepic}
\usepackage{url}
\usepackage{color}
\usepackage[utf8]{inputenc} 

\usepackage{enumerate}

\newtheorem{theorem}{\bf Theorem}[section]
\newtheorem{lemma}[theorem]{\bf Lemma}
\newtheorem{cor}[theorem]{\bf Corollary}

\newtheorem{problem}[theorem]{\bf Problem}
\newtheorem{prop}[theorem]{\bf Proposition}
\newtheorem{conj}[theorem]{\bf Conjecture}
\newtheorem{construction}[theorem]{\bf Construction}
\newtheorem{nota}[theorem]{\bf Notation}

\newtheorem{remark}[theorem]{\bf Remark}
\newtheorem{defi}[theorem]{\bf Definition}
\newtheorem{result}[theorem]{\bf Result}

\newcommand{\cut}[1]{}

\def\ex{\hbox{\rm ex}}

 \makeatletter
 \@addtoreset{equation}{section}
 \makeatother

 \makeatletter
 \@addtoreset{table}{section}
 \makeatother

\title{Supersaturation of $C_4$: from Zarankiewicz towards Erdős-Simonovits-Sidorenko}

\author{
  {\bf Zolt\'an L\'or\'ant Nagy}\thanks{The author is supported by the Hungarian Research Grant (OTKA) No. K. 120154 and by the J\'anos Bolyai Research Scholarship of the Hungarian Academy of Sciences}\\ \\ 
{\small MTA--ELTE Geometric and Algebraic Combinatorics Research Group}\\
{\small ELTE E\"otv\"os Lor\'and University, Budapest, Hungary}\\
{\small Department of Computer Science}\\
{\small H--1117 Budapest, P\'azm\'any P.\ s\'et\'any 1/C, Hungary}\\
{\small \tt{  nagyzoli@cs.elte.hu}}}
\date{}

\begin{document}
\maketitle

\begin{abstract}

 For a positive integer $n$, a graph $F$ and a bipartite graph  \ $G\subseteq K_{n,n}$   let ${F(n+n, G)}$ denote the number of copies of $F$  in $G$, and let 
$F(n+n, m)$ denote the minimum number of copies of $F$ in all graphs $G\subseteq K_{n,n}$ with $m$ edges. The study of such a function is the subject of theorems of supersaturated graphs and closely related to the Sidorenko-Erdős-Simonovits conjecture as well.
In the present paper we investigate the case when $F= K_{2,t}$ and in particular the quadrilateral graph case. For $F=C_4$, we obtain exact results if $m$ and the corresponding Zarankiewicz number differ by at most $n$, by a finite geometric construction of almost difference sets. $F= K_{2,t}$ if  $m$ and the corresponding Zarankiewicz number differs by  $Cn\sqrt{n}$ we prove asymptotically sharp results. We also study stability questions and point out the connections to covering and packing block designs.

  \bigskip\noindent \textbf{Keywords}: extremal graphs, supersaturation, quadrilateral, almost difference sets, designs, finite geometry 
\end{abstract}

\section{Introduction}

To determine the minimal number of a  subgraph $F$ in certain graph families is one of the oldest problems in combinatorics. In fact, it dates back to the  results of Mantel and Tur\'an \cite{survey}, who  investigated  the case of cliques ($F=K_p$) to describe the maximal number of edges of $n$-vertex graphs for which the   minimal number of subgraphs $F$ can still be zero. These results initiated the study of the function $$\ex(n, F)=\max\{|E(G)|: |V(G)|=n, F\not\subseteq G\},$$ and the so called \textit{extremal graphs} for which the number of edges meet the extremal function;  furthermore in general, extremal graph theory.

  Let us denote by $F(n,G)$ the number of $F$-subgraphs in the $n$-vertex graph $G$, and denote by $F(n,m)$ the minimal number of $F$-subgraphs in $n$-vertex graphs having $m$ edges. Clearly, $F(n,m)=0$ for $m=\ex(n, F)$, while $F(n,\ex(n, F)+1)>0$ holds. Following the pioneer unpublished result of Rademacher, Erd\H{o}s and Simonovits \cite{ES83} started to investigate the order of magnitude of the function $F(n,m): m>\ex(n, F)$ for arbitrary $F$  and called theorems which asserts that a graph $G=G_n$ contains very many graphs $F$ from a family $\mathcal{F}$  \textit{ theorems on supersaturated graphs}. Such
theorems are not only interesting in themselves, but also are often useful in establishing other extremal results.

If the number of edges have form $m=	\ex(n, F)+k$, then we call $k$ the \textit{excess} (over the extremal number). 

In the case $F=K_{p+1}$,  Erd\H{o}s \cite{E62} proved a stability result when the excess is small.

\begin{theorem}
For every fixed $p$ there exists a constant $c_p>0$ such that in any graph $G$ on $n$ vertices and $m= e(T_{n,p})+k$ edges (where $k<c_pn$) contains as many $K_{p+1}$ as the graph $T_{n,p}^{(+k)}$, where $T_{n,p}^{(+k)}$ is obtained from the Tur\'an graph $T_{n,p}$ by putting $k$ edges in one of its maximal classes so that the new edges form no triangle.
\end{theorem}

The following stability theorem became a  milestone in the study of supersaturated graphs.

\begin{theorem}[Lovász-Simonovits, \cite{LS75}]\label{LS}
	Let $C$ be an arbitrary constant
and let  $p$ be a fixed integer. There exists two  constants $\delta>0$ and $C'$ such that in any graph $G$ on $n$ vertices and $m= e(T_{n,p})+k$ edges ( $0<k<\delta\cdot n^2$), if the number of copies of $K_{p+1}$ in $G$ is at most $Ckn^{p-1}$ then $G$ can be obtained from the Turán graph  $T_{n,p}$ by adding or deleting only at most $C'k$ edges.
\end{theorem}

In general, the cliques can be replaced by any graph $F$ in the problem. It turns out that the analogue of Theorem \ref{LS} holds for every graph $F$ which has chromatic number $p+1>2$, due to  Erdős and Simonovits \cite{ES83, survey}. The result is based on the application of the celebrated Erdős--Stone--Simonovits theorem.\\
However if  $F$ is bipartite, the situation is significantly different, since the extremal numbers of bipartite graphs have exponents smaller than $2$.

Concerning the $F$ bipartite case, Erdős and Simonovits also proved 

\begin{theorem}[Erdős-Simonovits, on the number of complete bipartite graphs \cite{ES83, survey}]\label{Kab}	\label{es}
  For any $a\leq b$,   there exist two constants $c, \gamma > 0 $ such that if an $n$-vertex  graph $G$ has  $m>cn^{2-\frac{1}{a}}$ edges,
then $G$ contains at least $$ \frac{\gamma\cdot m^{ab}}{n^{2ab-a-b}}$$
 copies of $F=K_{a,b}$.
\end{theorem}

\begin{defi}
We call $  c, \gamma= \gamma(c) > 0 $ the {\em density parameters} for $K_{a,b}$.
\end{defi}

\begin{cor}[\cite{survey}]\label{Erdo}
 Let $c > 0$ be an arbitrary constant and  $G$ be an $n$-vertex  with  $m$ edges. If $m > c\cdot \ex(n, C_4)$, then $G$
contains at least $\gamma m^4/n^4$
copies of $C_4$, for some $\gamma > 0$. 
\end{cor}

As it was noted in their paper, the random graph with $m$ edges shows that this bound is sharp, which is a natural phenomenon in these kind of problems. Two fundamental questions arise here. The first is to determine the exact value of $\gamma$ depending on $c$. The sequent one is to find the largest number $c'$ for which the theorem  does not hold, and to characterize the number of copies of $F$ appearing in graphs having $\ex(n,F)\leq m\leq (c'+o(1))n^{2-\frac{1}{a}}$ edges.

Notice that the first question has a strong  connection to the famous Sidorenko conjecture \cite{Sido} which was proposed independently in a slightly different form by Erdős and Simonovits \cite{ES83}. We consider a discrete form similar to that of Erdős and Simonovits.

\begin{conj}[Sidorenko, Erdős-Simonovits \cite{ES83}]\label{sid} Let $H$ be a bipartite graph with vertex sets $a_1,\dots,a_r$ and $b_1,\dots,b_s$. Let the number of edges of $H$ be $m_0$. Let $G$ be a bipartite graph of density $\alpha$ with vertex sets $X$ and $Y$ and let $\phi$ and $\psi$ be random functions from $\{a_1,\dots,a_r\}$ to $X$ and from $\{b_1,\dots,b_s\}$ to $Y$. Then the probability that $\phi(a_i)\psi(b_i)$ is an edge of $G$ for every pair $i,j$ such that $a_ib_j$ is an edge of $H$ is at least $\alpha^{m_{0}}$.
\end{conj}

Note that in this variant, we not only consider the copies of $H$ but also its images under homomorphisms (edge-preserving mappings) with appropriate weights. This difference between the statement of Conjecture \ref{sid} and Theorem \ref{es} can be essential only if the host graph $G$ is not dense.

%For $H=C_4$, this is just an easy application of a double counting and the Cauchy-Schwarz-inequality.
The Sidorenko-Erdős-Simonovits conjecture was confirmed for several bipartite graph families, including complete bipartite graphs (see for instance \cite{Conlon, Fox, Sido2, Szegedy}), but is still widely open. One can observe here that in the \textit{dense case}, when $m=\Omega(n^2)$, the density parameters happen to be the same. Also, this conjecture reveals the reason why Theorem \ref{Kab} was restricted to complete bipartite graphs: the reason was not only our lack of knowledge on the exact value of $\ex(n,F)$ in general for bipartite $F$, but the structure of graphs with few copies of $F$ in dense graphs is not known as well.

Our aim is to improve the bound in Theorem \ref{Kab} in the \textit{sparse case}, that is, when $|E(G)|=o(n^2)$. In order to do that, from now on we consider only graphs $G$ which are subgraphs of a balanced bipartite graph $K_{n,n}$. We introduce some notations.

\begin{nota} $ $
\begin{itemize} \item $F(n+n,G)$ denotes the number of $F$-subgraphs in the bipartite graph $G$ on partition classes of size $n$.
 \item $F(n+n,m)$ denotes the minimal number of $F$-subgraphs a graph $G$ can have where \mbox{$G\subseteq K_{n,n}$} have $m$ edges. $d(i)$ denotes the degree of  a vertex $i$ in $G$
 %\item $w_2(j,k)$ denotes the number of walks of length $2$ from $j$ to $k$.
 \item $d(X):=d_G(X)$ denotes the co-degree of $X$, i.e. the number of  common neighbors of the vertices of $X\subseteq V(G)$ in $G$. For the degree of a single vertex $y$ we use the standard notion $d(y)$.
\end{itemize}
\end{nota}

\begin{remark} The definition of $F(n+n,m)$  implies that $$z(n,n,a,b)=\max \{ m \ |  \ K_{a,b}(n+n,m)=0 \}$$ holds for the well known Zarankiewicz-number, see \cite{survey}.
\end{remark}
 %https://gowers.wordpress.com/2015/11/18/entropy-and-sidorenkos-conjecture-after-szegedy/
 % https://arxiv.org/pdf/1406.6738v3.pdf,  https://arxiv.org/pdf/1510.06533v1.pdf

In this paper we mainly focus on the questions concerning the density parameters described above for the simplest non-trivial bipartite graph, the quadrilateral.\\ 
In the spirit of Theorem \ref{es} we establish a general lower bound on $C_4(n+n,m)$ in the next section, which points out the dependence of the density parameters $ c, \gamma$ as follows.

\begin{theorem}\label{magnitude} Suppose $m=n(\sqrt{n}+\frac{1}{2})+\xi(n)$ for $\xi(n)\geq 0$. Then

\begin{enumerate}[(i)]
	\item $ \xi(n)=O(\sqrt{n})$ implies $C_4(n+n,m)=\Omega({n})$;  
	
		\item $  \sqrt{n} \ll \xi(n) \ll n\sqrt{n}   $   \hspace{0,2cm} $C_4(n+n,m)\geq (\frac{1}{2}+o(1))\sqrt{n}\xi(n)$;   
	
	\item	$ \xi(n) =Cn\sqrt{n}$ implies $C_4(n+n,m)\geq \left(\frac{C(C+2)(1+C)^2}{4} +o(1)\right)n^2$;  
	
	\item	$  \xi(n) \gg  n\sqrt{n}$    implies   $C_4(n+n,m)=(1+o(1))\frac{1}{4}\left(\frac{m}{n}\right)^4$.  
	\end{enumerate}
	%$ m= z(n,n,2,2)+o(n^{2})$ \ \ \ $\f(n,m,C_4)=O(\sqrt{n})$  \\
\end{theorem}

 Then we point out the conditions under which the general lower bound is met, namely if there exists a so-called symmetric (almost) design with required parameters. The latter fact in turn implies that for supersaturated extremal problems concerning bipartite graphs, we cannot count on unique stability results similar to Theorem \ref{LS}, as several infinite families  of parameters $2-(v,k, \lambda)$ exist with many non-isomorphic block designs \cite{cont}.\\ % Bhat, Shrikhande- Non-isomorphic Solutions of Some Balanced Incomplete Block Designs. I (1970)
  In Section 3. we investigate the case when the excess is small, i.e. case $(ii)$ in Theorem \ref{magnitude}. It turns out that in this  case some sort of stability does exist, and it relies on some finite geometry argument which leads to the discovery of a new almost difference set family. Then in the next Section we examine the case when the excess is the same order of magnitude as the Zarankiewicz number $z(n,n,2,2)$ (case $(iii)$), and prove the following asymptotic result.
	
\begin{theorem} \label{ennyi} For any fixed positive integer $k$, if $m=(\sqrt{k}+o(1))n\sqrt{n}$, then 
$$\frac{C_4(n+n, m)}{n^2}\rightarrow \frac{k(k-1)}{4}$$ while the random balanced bipartite  graph contains $\frac{k^2}{4}n^2$ quadrilaterals.
\end{theorem}

We also extend this theorem for the number of graphs $K_{2,t}$.

Finally in Section \ref{4.0} we put our results into perspective and point out the connections of this problem  with clique-packing or covering problems in graph theory, discussing some open questions.

\section{Lower bound on the number of $C_4$s in terms of the number of edges}

%Note that the statement deals only with   the \textit{dense case} when $|E(G)|=\Omega(n^2)$.

\begin{theorem}[Theoretical lower bound]\label{theoLB}
Let $G$ be a subgraph of $K_{n,n}$ with $m$ edges on partite classes $X$ and $Y$. Then the number of $F=K_{a,b}$  subgraphs in $G$ is at least $$ \displaystyle\binom{n}{a}
\displaystyle\binom{{n\binom{\overline{d}}{a}}\cdot{\binom{n}{a}}^{-1}}{b},$$ provided $A\subseteq X$ and $B\subseteq Y$, where $\overline{d}=m/n$ is the average degree in $G$.

 Here we consider the truncated  extension of the binomials for which $\binom{n}{k}= \frac{\prod_{i=0}^{k-1}{(n-i)}}{k!}$ if $n\geq k\geq 0$, otherwise it is zero. Note that this function of $n$ is convex for fixed $k$ if $n>0$.
\end{theorem}

Although it is essentially the application of the proof idea of the Kövári-Sós-Turán theorem, we give a proof for the sake of completeness.
 
\begin{proof}[Proof]
	We may fix first a set of cardinality $a$ in $X$ and count the those sets in $Y$ which have exactly $b$ vertices, which are all adjacent to the vertices of the $a$-set; then we can sum this up to all $a$-sets. This provides 
	
	$$K_{a,b}(n+n,G)=\sum_{A\subseteq X, |A|=a} \binom{d(A)}{b}.$$
	
	By applying Jensen's inequality, we get
	
		$$K_{a,b}(n+n,G)\geq  \binom{n}{a}\binom{   {\sum_{A\subseteq X,|A|=a}{d(A)}}{\binom{n}{a}^{-1}}}{b}.$$
	
	Note that the sum of the co-degrees counts the way one can choose a vertex from $Y$ and $a$ neighbors of that vertex. Hence we obtain
	$$\sum_{A\subseteq X,|A|=a}{d(A)}= \sum_{y\in Y}\binom{d(y)}{a}\geq n\binom{n^{-1}\sum_y d(y)}{a}  $$  after applying Jensen's inequality again. This yields the stated expression.
\end{proof}

%Let us examine the order of magnitude of the theoretical lower bound in terms of the excess, for $F=C_4$. 

This result enables us to prove Theorem \ref{magnitude} to reveal  the connection between the density parameters $\gamma$ and $c$ in Corollary \ref{Erdo} for the graph $F=C_4$.

\begin{proof}[Proof of Theorem \ref{magnitude}] The result follows from Theorem \ref{theoLB}. Suppose that $m=n(\sqrt{n}+\frac{1}{2})+\xi(n)$ for $\xi(n)\geq 0$. Then the lower bound implies that 
	$$K_{a,b}(n+n,m)\geq \binom{n}{2}\cdot \frac{1}{2} \frac{m(m-n)}{n^2(n-1)}\left(\frac{m(m-n)}{n^2(n-1)}-1 \right).$$
	Substituting $m=n(\sqrt{n}+\frac{1}{2})+\xi(n)$ we obtain 
	
	$$K_{a,b}(n+n,m)\geq \frac{1}{4n^3(n-1)}m(m-n)   \left(m(m-n)- n^2(n-1) \right)=       $$
	
	$$\frac{1}{4n^3(n-1)}\left(n^3-n^2/4+2n\sqrt{n} \xi(n)+\xi(n)^2 \right)   \left(3n^2/4+2n\sqrt{n} \xi(n)+\xi(n)^2  \right)     $$
	
	Thus the order of magnitude of $\xi(n)$ determines the leading term in the bound, providing the result. In the case $(iv)$, balanced bipartite  random  graphs show that the lower bound is sharp.
\end{proof}

In brief, we get that if $m / z(n,n,2,2) \rightarrow \infty$, then
$C_4(n+n,m)$ indeed attained on the random graph apart from an smaller error term, while
if $m / z(n,n,2,2) \rightarrow (1+C)$ where $(C>0)$, then
$C_4(n+n,m)$ and the number of $C_4$s in the random graph are of same order of magnitude, but the two numbers not equal asymptotically.
Finally if $m / z(n,n,2,2) \rightarrow 1$, then
$C_4(n+n,m)$ is much smaller than the number of $C_4$s in the random graph.

%If $a=b=2$ that is, in the case of $F=C_4$, this simplifies to

%\begin{lemma}$$C_4(n+n,G)= \sum_{j,k\in B} \binom{w_2(j,k)}{2}\geq  \binom{n}{2}\binom{{\sum_{j,k\in B}{w_2(j,k)}}{\binom{n}{2}^{-1}}}{2}$$ $$\geq \binom{n}{2} \binom{n\binom{\frac{m/n}{2}}{2}\cdot \binom{n}{2}^{(-1)}}{2}.$$
%\end{lemma}

\subsection{Conditions for sharp results}

The theoretical lower bound \ref{theoLB} can be a bit sharpened if we take into consideration that both  degrees and co-degrees are integer numbers, 
hence we could apply the following discrete variant of the Jensen inequality.

\begin{lemma}[Discrete Jensen inequality]\label{discJen}  
	Let $f: \mathbb{Z}\rightarrow \mathbb{R}$ be a convex function. Then for any set of $N$ evaluations,
	
	$$\sum_{i=1}^N f(x_i)\geq \alpha f\left(\left\lfloor\frac{\sum_i x_i}{N}\right\rfloor\right)+\beta f\left(\left\lceil\frac{\sum_i x_i}{N}\right\rceil\right),$$
	
	where $\alpha$ and $\beta$ are determined such that 
	$\alpha+\beta=N$ and
	$  \alpha \left\lfloor\frac{\sum_i x_i}{N}\right\rfloor+\beta \left\lceil\frac{\sum_i x_i}{N}\right\rceil = \sum_i x_i$.
\end{lemma}

If we apply Lemma \ref{discJen} instead of Jensen's inequality, we got a slightly stronger lower bound that is independent of divisibility conditions of the parameters. We apply it for the case $F= K_{2,2}$, and refer to it as \textit{improved theoretical lower bound}.

%\begin{proof}
%	We improve the bound of Theorem \ref{theoLB} on by estimating first $\sum_{y\in Y}\binom{d(y)}{2}$ then $\sum_{A\subseteq X, |A|=2} \binom{d(A)}{2}$ via Lemma \ref{discJen}.\\
%	Observe first that $\sum_{A\subseteq X,|A|=a}{d(A)}= \sum_{y\in Y}\binom{d(y)}{2} \geq  t\binom{d+1}{2}+(n-t)\binom{d}{2}=n\binom{d}{2}+td$.
	
%	Next, consider 
%\end{proof}

It is important to note that while the incidence graph of projective planes provides an infinite family of constructions where one can state sharp results for $a=b=2$, the Zarankiewicz problem is notoriously hard in general thus only bounds are known for $z(n,n,a,b)$ when $b\geq a > 2$. Notably, the bound gained from  the discrete Jensen inequality (see the paper of Roman \cite{Roman}) is not sharp for every value $n$ with  $a, b > 2$ fixed. This follows from the improvement by Füredi\cite{Fured} on the bound of the Zarankiewicz number $z(n,n,s,t)$, who proved that 
$$z(n,n,s,t)\leq ((b-a+1)^{1/a}+o(1))n^{2-1/a},$$ while the well known Kővári-Sós-Turán theorem closely related to Theorem \ref{theoLB} only give  \  $z(n,n,s,t)\leq((b-1)^{1/a}+o(1))n^{2-1/a}$.

Interestingly, it can be improved in the case $a=b=2$ as well for certain values of $n$ due to Damásdi, Héger and Szőnyi see  \cite{Heger}, and  the  bound on $z(n,n,a,b)$ in the same spirit of Theorem \ref{theoLB} turned out to be far from sharp even in this case for particular values of $n$. 

From now on, we focus on the conditions when a graph attains the theoretical lower bound for the case $K_{a,b}=K_{2,2}$ in Theorem \ref{theoLB}. The following statement is the straightforward consequence of the two estimations via Jensen's inequality in the proof.

\begin{cor}\label{general1} A  graph $G$ attains the bound of Theorem \ref{theoLB} for $K_{a,b}=K_{2,2}$ if and only if
\begin{itemize}
	\item $G$ is   regular;
	\item every subset  $A\subseteq X$ of size $a$ has  the same number of common neighbors, that is, $||d(A)|-|d(A')||=0$ $ \forall A, A'\subseteq X, |A|=|A'|=a$.
\end{itemize}
\end{cor}

Observe that conditions of the structure of graph attaining the improved theoretical lower bound with  Lemma \ref{discJen}  can be derived  similarly.

\begin{cor}\label{general2} A  graph $G$ attains the {\em improved theoretical lower bound} for $K_{2,2}$ if and only if
	\begin{itemize}
		\item $G$ is  almost regular, that is,  
		$|d(v)-d(v')|\leq 1$ \ \ \ ($\forall v,v' \in V(G)$);
		\item every subset  $A\subseteq X$ of size $a$ has almost the same number of common neighbors, that is, $||d(A)|-|d(A')||\leq 1$ $ \forall A, A'\subseteq X, |A|=|A'|=a$.
	\end{itemize}
\end{cor}

\subsection{Case of equality with the theoretical lower bound} \label{2.0}

Here we make the connection between balanced bipartite graphs attaining the theoretical lower bound and symmetric block designs or symmetric almost designs.
We recall the definitions of these structures, together with some useful observations which will be essential later on.

\begin{defi}[Design] $\mathcal{D(P,B)}$ is a \textit{symmetric $2-(v,k,\lambda)$ block  design} if \\
 $\mathcal{P}$ is a set of points,   $\mathcal{B}$ is a set of $k$-uniform blocks, with\\
 $\bullet$ $|\mathcal{P}|=v=|\mathcal{B}|$, and \\ $\bullet$  for every pair of points $ \{x,y\}\subset \mathcal{P}$ the number of incident blocks to this pair  $|B_i\in\mathcal{B}: x,y\in B_i|=\lambda$.
\end{defi}

Note that the equality $v\binom{k}{2}=\binom{v}{2}\lambda$  connects the parameters of  symmetric designs.
Corollary \ref{general1} directly implies the following.

\begin{cor}[Every design provides equality] In the incidence graphs of a symmetric $2-(v,k,\lambda)$ block  designs we have\\
$\bullet$ $|d(i)-d(i')|=0$ \ \ \ $( \forall i,i' \in V(G) )$;\\
$\bullet$ $|d(\{j,k\})-d(\{j', k'\})|=0$ \ \ \ $ ( \forall \{j,k\}, \{j',k'\} \in \binom{V(G)}{2} )$, thus these graphs provide equality in Theorem \ref{theoLB}
\end{cor}

A bit more general concept is the so-called almost design or \textit{adesign}.

\begin{defi}[Adesign] $\mathcal{D(P,B)}$ is a  \textit{symmetric $2-(v,k,\lambda)$ block adesign} (almost design) if
$\mathcal{P}$ denotes a set of points,   $\mathcal{B}$ is a set of $k$-uniform blocks, with \\$\bullet$ $|\mathcal{P}|=v=|\mathcal{B}|$, and\\ $\bullet$  for every pair of points $\{x,y\}\subset \mathcal{P}$ the number of incident blocks to this pair   $|B_i\in \mathcal{B}: x,y\in B_i|$ equals $\lambda$ or $\lambda+1$.
\end{defi}

Similarly to the designs, it is easy to see that the more general Corollary \ref{general2} implies

\begin{cor}[Every Adesign provides equality]\label{ady} In the  incidence graphs of symmetric $2-(v,k,\lambda)$  adesigns we have\\
$\bullet$ $|d(i)-d(i')|=0$ \ \ \ $( \forall i,i' \in V(G) )$;\\
$\bullet$ $|d(\{j,k\})-d(\{j', k'\})|\leq 1$ \ \ \ $ ( \forall \{j,k\}, \{j',k'\} \in \binom{V(G)}{2})$, thus these graphs provides equality in the improved theoretical lower bound.
\end{cor}

Various infinite families of designs and adesigns exist, see \cite{cont} for recent surveys.
For our purposes, we describe a special family of designs, which are connected to difference sets.

\begin{defi}  A $(v,k,\lambda)$ \textit{difference set} is a subset $D$ of size $k$ of (an additive) group $G$ of order $v$ such that every non-zero element of $G$ can be expressed as a difference $d-d'$  with a pair of elements of $D$ in exactly $\lambda$ ways. A difference set $D$ is said to be cyclic or Abelian if the group $G$ has the corresponding property. 
\end{defi}

A difference set with $\lambda = 1$ is  called \textit{planar}, and it is well known due to Singer that there always exist planar difference sets  for $G=\mathbb{Z}_n$ with $n=q^2+q+1$ where $q$ is some power of a prime, as cyclic projective planes provide such examples. For further details, we refer to \cite{cont, diff}.

\begin{defi}  A $(v,k,\lambda)$ \textit{almost difference set} is a subset $D$ of size $k$ of (an additive) group $G$ of order $v$ such that every non-zero element of $G$ can be expressed as $d-d'$  with a pair of elements of $D$  in exactly $\lambda$  or $\lambda+1$ ways, with both cases appearing.
\end{defi}

\begin{remark} If $D$ is an (almost) difference set, and $g\in G$, then $-D=\{-d:d\in D\}$ is also an (almost) difference set, furthermore $g+D=\{g+d:d\in D\}$ is also an (almost) difference set, and is called a translate of $D$. The set of translates of $D$ yields a  $2-(v,k,\lambda)$ design (or adesign).
\end{remark}

A key concept concerning difference sets is the multiplier.

\begin{defi} A multiplier of a difference set $D$ in group $G$ is a group automorphism $ \phi$  of $G$ such that $D^{\phi }=g+D$ for some $ g\in G$. If $G$ is abelian and $\phi$  is the automorphism that maps $ h\rightarrow  t\cdot h$, then $t$ is called a (numerical)  multiplier.
\end{defi}

% https://en.wikipedia.org/wiki/Difference_set  

% at kell szamolni, figyelni kell ra hogy most a korrelacio erosebb lett...

\section{$F=C_4$, when the excess is small - connection to finite geometries}

In this section we prove that slightly above the Zarankiewicz number (in cases (i) and (ii) of Theorem \ref{magnitude}) it is possible to attain the improved lower bound by adding edges to the incidence graph $G(\Pi_q)$ of a projective plane $\Pi_q$. On the other hand, this extension method only works in a rather small domain. Finally we describe the incidence structure of points and lines corresponding to the added edges. For an introduction on finite geometries, we refer to \cite{ball}.

\begin{theorem}\label{completing} If $n=q^2+q+1$ for  a prime power $q=p^h$ and $m\leq z(n,n,2,2)+n$, then $C_4(n+n,m)$ meets the improved theoretical lower bound. 
\end{theorem}

\begin{proof} It is clearly enough to show a construction for $m= z(n,n,2,2)+n$ which contains the incidence graph of a projective plane $\Pi_q$ of order $q$. To this end, we will add a matching to $G(\Pi_q)$  between non-incident points and lines in the following way. Consider a planar difference set $D$ corresponding to  a cyclic projective plane of order $q$. We would like to add another element $g$ of $\mathbb{Z}_n$ to $D$ such that $D\cup g$ is an almost difference set, i.e. every non-zero element appears once or twice as a difference. Observe that $g$  cannot be added if only if $g$ is already an element of $D$ or $d-g=g-d' \Leftrightarrow d+d'=2g$ for a pair of different elements in $D$. Hence $n-|D|-\binom{|D|}{2}=\binom{q}{2}$ elements can complete $D$ to an almost difference set $D^*$ of size $q+2$. We call such elements \textit{completion elements} (w.r.t. $D$).\\
The incidence graph of the adesign formed by the translates of $D^*$ thus provide such an example in view of Corollary \ref{ady}.
\end{proof}

\begin{remark} The constructed $D\cup {g}$  provides a new infinite family of almost difference sets with parameters $(q^2+q+1, q+2, 1)$.
\end{remark}

\begin{remark} Note that during the completion, we used the fact that $2\nmid n$. Also, a slightly more involved calculation shows that any difference set with $\lambda=2$ can be completed by an element to an almost different set.
\end{remark}

In the geometric point of view, as the translates of the planar difference set $D$ are corresponding to the lines of the plane, one might ask about the structure of the points that can be added to the lines to obtain the desired almost difference set with its translates. %Note that lines and certain $q+1$-arc have a strong connection in cyclic projective planes, as if $\mathcal{L}$ is line in  $\mathbb{Z}_n$, then $-\mathcal{L}$ will be an oval due to a theorem of Kiss, Malni\v{c} and Maru\v{s}i\v{c} \cite{Kiss}.  Observe also that the number $\binom{q}{2}$ coincides with the so called internal points of an oval ($q+1$-arc) of a plane of odd order. \\
This geometric structure of the completion elements is described in the following theorems.

\begin{theorem}\label{hyper} If the order of the projective  plane is even, then  the completion elements of the difference set $D$ are the points not incident to a dual hyperoval. 
\end{theorem}

\begin{theorem}\label{bundle} If the order of the projective  plane is odd, then  the completion elements of the difference set $D$ are the points not incident to a set of $q+1$ ovals which pairwise intersect each other in exactly one point.
\end{theorem}

Recall that an oval, resp. hyperoval is a set of $q+1$, resp. $q+2$ points in the projective plane of order $q$ no three of which are collinear. 
This also implies that any line intersects the hyperoval either in $0$ or $2$ points, hence  hyperovals are also maximal $(k,n)$ arcs as well. Hyperovals only exist if the order is even, and in that case every oval ($q+1$-arc) extends uniquely to a hyperoval. (For further details, see \cite{ball}.) 

A key ingredient in the proofs is the following algebraic result due to M. Hall \cite{Hall}. 

\begin{result}[Hall \cite{Hall}, see also \cite{Jung}]  
Let $D$ be an Abelian difference set for a projective plane of order $q$.
Then $2$ and $1/2$ are (numerical) multipliers of $D$ if and only if $q$ is even. Moreover, $\frac{1}{2}D$  is the set of absolute points of a certain polarity, and it forms an oval if $q$ is odd or a line if $q$ is even.
\end{result}

\begin{proof}[Proof of Theorem \ref{hyper} and \ref{bundle}]
In the proof of Theorem \ref{completing}, we saw that  $g\in \mathbb{Z}_n$  is not a completion element if and only if  $2g=d+d'$ for a pair of (not necessarily distinct) elements $d, d'$ of $D$, hence these elements $g$ have the form $\frac{d+d'}{2}$. This implies that the set of such elements are determined by the  union of the some translates of $\frac{D}{2}$, namely $$\bigcup_{d\in D}\left\{\frac{d}{2}+ \frac{D}{2} \right\}.$$

Note that these translates intersect each other in exactly one element and no three of them share a common element. It follows from Hall's result that if $q$ is even, then these are translates of $D$ since $\frac{1}{2}$ is a multiplier, hence they are lines as well, which form a dual hyperoval together with $D$. \\
If $q$ is odd, then these are translates of an oval according to Hall's result, again with no three of them share a common element and no two of them have more than one common element, completing the proof. We note that such a structure is called a projective bundle and have many interesting properties, see \cite{Jung} and references therein.
\end{proof}

Via Theorem \ref{completing} one could add $n$ edges to the incidence graph of the projective plane many ways to get a graph which attains the improved theoretical lower bound corresponding to the number of $C_4$s. However, if the excess over $z(n,n,2,2)$ is more than $n$, this is not the case.

\begin{prop} Suppose that  $n= q^2+q+1$, then $C_4(n+n,G)$ cannot meet the improved theoretical lower bound on graphs $G$ with $m$ edges containing $G(\Pi_q)$ as a subgraph if $z(n,n,2,2)+n< m \leq \sqrt{2}\cdot z(n,n,2,2)$.
\end{prop}

\begin{proof} Suppose to the contrary that there exists such a construction with excess more than $n$. The edges not contained in $G(\Pi_q)$  can be considered as newly introduced incidences between lines and points. Since the excess is more than $n$, there exists a point $P$ which is now joined with two lines $e$ and $f$ for which $P\not\sim e$ and  $P\not\sim f$ in $\Pi_q$. Lines $e$ and $f$ have a unique intersection $Q=e\cap f$ in $\Pi_q$, while $P$ and $Q$ determine a unique line $g$ in $\Pi_q$, different from $e$ and $f$. Consequently, $N(P,Q)\supseteq  \{e,f,g\}$ in the construction, which can only happen in an extremal construction if every point pair share at least two common lines, thus the statement follows.
\end{proof}

We remark in the end of this section that a recent result of Ferber, Hod, Krivelevich and Sudakov \cite{Ferber} proved the existence of certain  $2-(n,k,1)$ almost designs. However, their result cannot be applied to our case since they assume that 
the number of blocks is much larger than the number of points, but it could be applied for $F(n,n', m)$ to determine the  minimal number of $C_4$s in $m$-edge subgraphs of very much unbalanced bipartite graphs  $K_{n,n'}$ $(n\ll n')$.

%For an element $g \in  GF(q^3)*$ denote by $\left\langle x \right\rangle$ the point of
%$PG(2, q)$ identified with the $GF(q)$-subspace generated by $x$.
%This notation  implies $\left\langle x \right\rangle=\left\langle x' \right\rangle$ $\Leftrightarrow$ $x=\lambda x'$ for some $\lambda \in GF(q)$, i.e.
%$x^{q-1}=y^{q-1}$.

%Let $g$ be a primitive element of $GF(q^3)$. The collineation group
%generated by $\left\langle x \right\rangle \rightarrow \left\langle gx \right\rangle$ is a cyclic Singer group $G$ of
%$PG(2,q)$, that is a cyclic collineation group of order
%$q^2+q+1$ permuting the points in one orbit. Moreover it
%also permutes the lines in one orbit.

\section{When the excess is of order $n\sqrt{n}$ - generalizing the ideas of Mörs}

In his paper \cite{Mors}, Mörs proved an asymptotic result for the Zarankiewicz number $z(n,n,2,k)$. In this subsection we show a family of graphs $G$ for which $K_{2,2}(n+n, G)= K_{2,2}(n+n, m)$ for $m= O(n\sqrt{n})$. The idea behind the construction is the same, namely to exploit the additive and multiplicative structure of a finite field $\mathbb{F}_q$, but the construction below is defined in a much simpler manner and also more general. This idea appeared also in the paper of Füredi \cite{fredi} concerning $ex(n, K_{2,t})$. 

\begin{construction} \label{mors}
Let  $G$ be a bipartite graph on $V_1\cup V_2$ with edge set $E:=E(G)$ as follows.\\
$V_1=V_2= A \times B$  where  $A\sim \mathbb{F}_q$ and  $B= \{1, 2, \ldots, \frac{q-1}{k}\}$, furthermore  
\begin{equation}\label{edges}
\{ (a, b), (\alpha, \beta) \}\in E \ \Leftrightarrow \ \exists j\in \mathbb{Z}_k : g^{\beta}a+g^{b}\alpha=g^{j\cdot \frac{q-1}{k}}  \mbox{ \ in \  $\mathbb{F}_q$}
\end{equation} for a fixed primitive root $g \in \mathbb{F}_q^{\times}$. Here $(a, b)\in V_1$ and  $(\alpha, \beta)\in V_2$.\\
We refer to this graph $G$, depending on parameters $q$ and $k$ as $G^{(q,k)}$
\end{construction}

\begin{prop}\label{elek} The graph in Construction \ref{mors} is regular and 
	$E(G^{(q,k)})=\frac{q(q-1)^2}{k}$.
\end{prop}

\begin{proof} Our aim is to prove that for every vertex $v=(a,b) \in V_1$, $d(v)= q-1$.
Pick an arbitrary vertex $(a,b)$.
The defining equation \ref{mors} yields that for every choice $\beta \in B, j\in \mathbb{Z}_k $, exactly one $\alpha \in A$ fulfills the condition. The proposition follows if we observe that different pairs $(\beta, j), (\beta, j')$ cannot determine
the same $\alpha \in A$.
%To this end, consider the equation $$g^{\beta}a+g^{b+j\cdot \frac{p-1}{k}}=\alpha=g^{\beta}a+g^{b+j'\cdot \frac{p-1}{k}},$$ which implies  $ g^{j\cdot \frac{p-1}{k}}= g^{j'\cdot \frac{p-1}{k}}$. As $0<|\frac{p-1}{k}(j-j')|\leq \frac{p-1}{k}(k-1)<p-1$, this would be a contradiction.
\end{proof}

\begin{prop}\label{c4ek} The co-degree for any pair of vertices is either $0$ or $k$  in $G^{(q,k)}$, and 
$K_{2,2}(\frac{q(q-1)}{k}, G^{(q,k)})=\frac{q(q-1)^3}{4}\left(1-\frac{1}{k}\right)$. 
\end{prop}

\begin{proof}
We determine the co-degree of each pair $v= (a,b), v'=(a', b')$ from $V_1$.

\textbf{\textit{Case 1}. $a=a'=0$, $b\neq b'$.} In that case, we look for the possible solutions of the system of equations
\begin{align*}
  \begin{cases}
   \exists \ \ j\in \mathbb{Z}_k: \  g^{b}\alpha \ \ =   g^{j\cdot \frac{q-1}{k}}  \\
   \exists \  j'\in \mathbb{Z}_k: \  g^{b'}\alpha=  g^{j'\cdot \frac{q-1}{k}} 
  \end{cases}
\end{align*}

A solution $(\alpha,\beta)$  is called \textit{admissible}, if $\alpha\in A$ and $\beta\in B$ are both satisfied. 
Clearly $\alpha \neq 0$ must hold, but  that yields $ g^{b-b'}=  g^{(j-j')\cdot \frac{q-1}{k}}$, which is impossible, so the co-degree is $0$ in this case.\\

\textbf{\textit{Case 2.} $a\neq 0$.} We can rewrite the system of equations as

\begin{align*}
  \begin{cases}
   \exists \ \ j\in \mathbb{Z}_k: \  g^{\beta}a'+\frac{a'}{a}g^{b}\alpha \ \ = \frac{a'}{a}g^{j\cdot \frac{q-1}{k}}  \\
   \exists \  j'\in \mathbb{Z}_k: \  g^{\beta}a'+g^{b'}\alpha \ \ \ \ =  g^{j'\cdot \frac{q-1}{k}}.
  \end{cases}
\end{align*}

Subtracting the two equations we get  $$\left(\frac{a'}{a}g^{b}- g^{b'}\right)\alpha =\frac{a'}{a}g^{j\cdot \frac{q-1}{k}}-g^{j'\cdot \frac{q-1}{k}}.$$

\textbf{\textit{Case 2A.}} $\frac{a'}{a}g^{b}- g^{b'}=0$. This implies $\frac{a'}{a}= g^{j^*\cdot \frac{q-1}{k}}$ for some $j^*\in \mathbb{Z}_k$. Comparing this remark with the subcase condition, we get a contradiction, as $\frac{a'}{a} = g^{b'-b}$, however to obtain an admissible pair $b, b'$ we must have $\frac{a'}{a}=1$, but then $b'=b$ would also follow.

\textbf{\textit{Case 2B.}} $\frac{a'}{a}g^{b}- g^{b'}=\gamma \neq 0$. This allows as to express $\alpha$ as  \begin{equation} \label{eq:1} \alpha= \frac{1}{\gamma}\left( \frac{a'}{a}g^{j\cdot \frac{q-1}{k}}-g^{j'\cdot \frac{q-1}{k}} \right),\end{equation} and then we obtain $g^{\beta}$ as  \begin{equation} \label{eq:2} g^{\beta}=  (g^{j'\cdot \frac{q-1}{k}}-g^{b'}\alpha)/a'. \end{equation}

Clearly, $\alpha$ and thus $\beta$ depends on the choice of $j$ and $j'$, so we have to guarantee that these solutions are admissible. We claim that for any fix $j, j'$, exactly one choice from $\{j+i, j'+i\} : i\in \mathbb{Z}_k$ provides an admissible pair. Indeed, if $\alpha$ and $g^\beta$ is determined by $j$ and $j'$, then $j+i$ and $j'+i$ give 
$g^{i\cdot \frac{q-1}{k}}\alpha$ and $g^{i\cdot \frac{q-1}{k}}g^{\beta}$, hence this set gives exactly one admissible candidate for $\beta$. 

Observe on the other hand, that these admissible pairs provided by some $j$ and $j'$ are different. Indeed, a fixed $\alpha$ and $g^\beta$ determines exactly $j'$ (in the
equation \ref{eq:2}) and then together with $j'$, $j$ is also determined exactly (in the
equation \ref{eq:1}). Thus all in all, we get that there exists exactly $k$ common neighbors for $v= (a,b), v'=(a', b')$ in Case 2B.

This gives \begin{equation} \label{eq:r} \frac{1}{2}|V_1|\cdot \left(|V_1|-\frac{q-1}{k}\right)\binom{k}{2}=\frac{q(q-1)^3}{4}\left(1-\frac{1}{k}\right)
\end{equation}
 $K_{2,2}$ subgraphs, taking into consideration that for every $(a,b)$, exactly $\frac{q-1}{k}$ different other vertex from $V_1$ has the property that their pairwise co-degree with $(a,b)$ is zero according to Case 1 or Case 2A.
\end{proof}

\begin{remark}    
	
	It is easy to see that if the edge set of $G^{(q,k)}$ was defined by 
	\begin{equation}\label{edges2}
	\{ (a, b), (\alpha, \beta) \}\in E \ \Leftrightarrow \ \exists j\in \mathbb{Z}_k : g^{\beta}a+g^{b}\alpha=g^{\delta} g^{j\cdot \frac{q-1}{k}}  \mbox{ \ in \  $\mathbb{F}_q$}
	\end{equation}
	for some $\delta\in \mathbb{Z}$, then the proof of Proposition \ref{c4ek} would work out exactly the same way.  
\end{remark}

Comparing Proposition \ref{elek}, \ref{c4ek} to Theorem \ref{magnitude}, the proof of Theorem \ref{ennyi} easily follows.

\begin{proof}[Proof of Theorem \ref{ennyi}]
Choose a prime $p$ for which $p\equiv 1 \pmod {k}$. Then Construction \ref{mors} on $n+n= 2\frac{p(p-1)}{k}$ vertices  contains  $\frac{p(p-1)^2}{k}= \sqrt{k}\sqrt{\frac{p-1}{p}}n\sqrt{n}$ edges and $\frac{p(p-1)^3}{4}\left(1-\frac{1}{k}\right)= \frac{\sqrt{k}^2(\sqrt{k}+1)(\sqrt{k}-1)}{4}\frac{p-1}{p} n^2$. This shows that the construction provides asymptotically sharp result for every fix $k\in \mathbb{Z}$ if we take into consideration the density theorem of $1 \pmod k$ primes due to Huxley \cite{hux}.
\end{proof}

In Proposition \ref{c4ek} we prove that there exists exactly $k$ common neighbors for a pair $v, v' \in V_1$ which satisfies the condition of Case 2B, otherwise their co-degree is zero. This enables us to generalize the statement of Theorem \ref{ennyi} to provide asymptotic results on the number of $K_{2,t}$ graphs in a similar manner.

\begin{theorem}\label{k2t}
For any fixed positive integers $t>2$ and $k$, if $m=(\sqrt{k}+o(1))n\sqrt{n}$, then  
 $$\frac{K_{2,t}(n+n,m)}{n^2}\rightarrow \binom{k}{t}.$$
\end{theorem}

\begin{proof} From the theoretical lower bound Theorem \ref{theoLB} one can derive that either 
	${K_{2,t}(n+n,m)}\geq 0$ in the case $k<t$ or is at least $2\binom{k}{t}\binom{n}{2}$, asymptotically. 
	Note here that Theorem \ref{theoLB} counted the copies of the subgraphs with the prescription that $K_{2,t}$ is an ordered injection to $K_{n,n}$, which makes no difference in the case when $t=2$, in contrary to the general case. This explains the factor $2$ in the formula.\\
	
	If $k<t$, the graph $G^{(q,k)}$ clearly contains no copy of $K_{2,t}$. On the other hand, Proposition \ref{c4ek} about the codegrees of  $G^{(q,k)}$ implies that we can derive 
		$${K_{2,t}(n+n,G^{(q,k)})}= \frac{q(q-1)^3}{k^2}\binom{k}{t}$$ just as we got  equation \ref{eq:r}. Then the result follows.
\end{proof}

\section{Concluding remarks and related problems}\label{4.0}

Theorem \ref{magnitude} together with Theorem \ref{k2t} suggest that in fact the following stronger version of the conjecture of Erdős-Simonovits and Sidorenko might also hold.
\begin{conj} For a bipartite graph $F$, if $m/ex_{bi}(n+n, F))\rightarrow \infty$ holds for the number $m$ of edges in a balanced bipartite graphs, then the random graph with $m$ edges  has in expectation the asymptotically smallest number of copies of $F$.
\end{conj}

Observe that this is even stronger than the following

\begin{conj}  (Erdős-Simonovits)\cite{cube, SMiki} For every bipartite graph $F$, there exist
constants $\beta\in (0, 1), c > 0$ and  $n_0$ such that graph $G$ on $n \geq n_0$ vertices
with $m \geq n^{2-\beta}$
edges contains at least $cn^{|V (F)|}p^{|E(H)|}$
copies of $F$ where
$p =\frac{m}{ \binom{n}{2}}$
is the edge density of $G$.
\end{conj}
Note that the expected number of copies of $H$ in a random graph $G(n, p)$
is roughly $\frac{1}{|Aut(F)|}n^{|V (H)|}p^{|E(H)|}$.
Therefore the conjecture above in some
sense asserts that random graphs contain (or, is close to containing) the
minimum number of copies of $F$.

Two obvious way to extend our results is to consider a sample graph $F$ from larger set of bipartite graphs, and to consider the case when the host graph is $K_n$ and not $K_{n,n}$. The latter case, which is strongly connected to the investigated bipartite case is considered to be slightly more involved in general. Nevertheless, in the case of $F= K_{2,t}$, the construction of Füredi \cite{fredi} provides a matching bound to the corresponding lower bound, in the case when the number of edges is of order $n\sqrt{n}$.

An interesting consequence of Theorem \ref{theoLB} is that while many extremal graph problems have stability versions, it is not quite the case for $F(n,m)$ in general. Indeed, several infinite families of block design parameters  are known, for which many non-isomorphic block designs exist.
This motivates the following problem.

\begin{problem} Bound the maximal edit distance between incidence graphs of non-isomorphic block designs with same parameters $2-(v,k,\lambda)$, in terms of their parameters.
\end{problem}

The celebrated result of Wilson (see in \cite{cont}) states that if $n$ is large enough, and some obvious divisibility conditions hold, then there exist block designs of parameters $2-(v,k,\lambda)$. 

However, $n$ should be much larger than $k^2$ in the theorem, so it cannot be applied in general to obtain incidence graphs of block designs (approximately) attaining the supersaturation number $C_4(n+n,m)$ with some prescribed $m$. To overcome this, one may aim for almost uniform hypergraphs  which cover every pair of vertices almost equal number of times.

 More exactly, it would be interesting to see how close can we get to a design structure with hypergraphs having hyperedge size mostly $k$, while the covering number of every pair is mostly $\lfloor n{\binom{k}{2}}^{-1}\rfloor $ or  $\left\lceil n{\binom{k}{2}}^{-1}\right\rceil$   in the most interesting case of Theorem \ref{magnitude}  when $k=\Theta(\sqrt{n})$.
  Packing and covering designs represents known particular cases. A good construction would follow from an optimal \textit{$2-(v,k,\lambda)$  packing or a $2-(v,k,\lambda)$ covering} in which every edge appears in at most or at least one block, respectively. (For details, we refer to \cite{cont, cov}). But so far, the corresponding well-known Schönheim and Johnson bound is not known to be  asymptotically tight in the region in view.

Similarly to the case of almost designs, it would also be interesting to investigate the existence of almost difference sets, or structures close to an almost difference set, as well. Consider the following construction.

%The incidence structure of a design corresponding to a difference set motivates the following construction, which yields a lower bound for $C_4(n+n,G)$ with a suitable set $A$ of $ \mathbb{Z}_n$.

\begin{construction}\label{csoport}
Let $\mathcal{G}$ be a finite Abelian group of order $n$, written additively. Take a subset $A\subset \mathcal{G}$ of size $k$. Consider the balanced bipartite graph $G(\mathcal{G}, A)$ on vertex classes $\{A+g | g\in \mathcal{G}\}$ and $\{g| g\in \mathcal{G}\}$, where edges correspond to inclusion.      
%be a balanced bipartite graph on $n+n$ vertices, defined as follows. 
\end{construction}

Then for fixed $n$ and $k$, the closest $A$ is to be a difference set, the fewest copy of $C_4$ is determined by Construction \ref{csoport}.
More precisely, let $h_t(\mathcal{G}, A)$ be defined as 
$$h_t(\mathcal{G}, A)=\sum_{0\neq g\in \mathcal{G}} \left(|(a,a'): a-a'=g, a,a'\in A |\right)^t.$$

Then one can derive the following statement.

\begin{prop} If $n$ is odd, then $C_4(n+n, G(\mathcal{G}, A))=\frac{n}{4}\left(h_2(\mathcal{G}, A)-h_1(\mathcal{G}, A) \right)$
\end{prop}

The condition on the parity of $n$ follows from the fact that we must distinguish between the cases when $g$ and $-g$ can actually coincide or not for some elements of 
$\mathcal{G}$, but similar formula can be derived in the even order case too.

Observe that $h_1(\mathcal{G}, A)$ is fix if we prescribe $n$ and $k$, thus actually minimizing \mbox{$C_4(n+n, G(\mathcal{G}, A))$} is equivalent to minimize 

$$\Psi_2(\mathcal{G}, A):=\sum_{0\neq g\in \mathcal{G}}\left(|(a,a'): a-a'=g, a,a'\in A |- \frac{k(k-1)}{n-1}\right)^2 $$
Note that $\frac{k(k-1)}{n-1}$ is simply the average occurrence of a difference, thus $\Psi_2(\mathcal{G}, A)=0$ if and only if $A$ is a difference set.

This proposition motivates the following problem.

\begin{problem} Prove a general upper bound on $\displaystyle \min_{A\subset \mathcal{G}, |A|=k} \Psi_2(\mathcal{G}, A)$.
\end{problem}

%\begin{question}
%Given $k<n$, how can one bound the discrepancy of a multiplicities of the elements in the multiset  $A-A$ for  a  set $A\subset \mathbb{Z}_n$, $|A|=k$. In other words, how can one fill in the theorem: there exists a set $A$ for which $$\sum_{g\in \mathbb{Z}_n} \left(|(a,a'): a-a'=g, a,a'\in A |- \frac{k(k-1)}{n-1}\right)^2 \leq h(n,k),$$ for some error term $h(n,k)$.

%Here $m(g)=|(a,a'): a-a'=g, a,a'\in A |$ is the multiplicity of an element $g$ in the multiset $A-A$.
%\end{question}

Putting into perspective, the supersaturation problem we studied  can be also seen as  problem where we search for the extreme number of graphs $C_4$ or $K_{2,t}$ in graphs without  odd cycles and given edge density, which fits in as  a particular case of a recently widely investigated  general framework.

\begin{problem} Determine the largest number of subgraphs $F$ in graphs of order $n$ not containing a forbidden family $\mathcal{F}$ of graphs.
\end{problem}

The most notable recent results in the direction concerns the case $F=K_3$ and  $\mathcal{F}=\{C_5\}$ \cite{Alon, Gyori}, and  $F=K_m$ and  $\mathcal{F}=\{K_{s,t}\}$ \cite{Alon}. It has also extensions in extremal hypergraph theory as well \cite{Jiang, Loh}.

\textbf{
Acknowledgement}

The author would like to thank Bence Csajbók and Francesco Pavese for their remarks concerning Theorem \ref{hyper}.

\end{document}